\documentclass[a4paper]{amsart}
\usepackage{graphicx}
\theoremstyle{plain}
  \newtheorem{thm}{Theorem}[section]
  \newtheorem{lem}[thm]{Lemma}
  
  \newtheorem{prop}[thm]{Proposition}

\theoremstyle{definition}
  
  \newtheorem{ex}[thm]{Example}

\theoremstyle{remark}
  \newtheorem{rem}[thm]{Remark}
  \newtheorem*{ack}{Acknowledgments}
\newcommand{\Z}{\mathbb{Z}}
\newcommand{\cyclic}[1]{\Z/#1\Z}
\newcommand{\cyclicproduct}[2]{(\cyclic{#1})^{#2}}

\newcommand{\trivial}{\emptyset}

\newcommand{\vorbit}[1]{[\vec{#1}\,]}
\newcommand{\cvector}[1]{\vec{#1}}
\newcommand{\gvector}[1]{\overrightarrow{#1}}
\newcommand{\svector}[2]{\vec{#1}_{#2}}

\DeclareMathOperator{\complength}{w}

\numberwithin{equation}{section}
\allowdisplaybreaks
\begin{document}
\title{Homotopy invariants of Gauss phrases}
\author{Andrew Gibson}
\address{
Department of Mathematics,
Tokyo Institute of Technology,
Oh-okayama, Meguro, Tokyo 152-8551, Japan
}
\email{gibson@math.titech.ac.jp}
\date{\today}
\begin{abstract}
Equivalence relations can be defined on Gauss phrases using
 combinatorial moves.
In this paper we consider two closely related equivalence
 relations on Gauss phrases, homotopy and open homotopy.
In particular, in each case, we define a new invariant and determine the
 values that it can attain.
\end{abstract}
\thanks{This work was supported by a Scholarship from the Ministry of
Education, Culture, Sports, Science and Technology of Japan.} 
\maketitle
%%%%%%%%%%%%%%%%%%%%%%%%%%%%%%%%%%%%%%%%%%%%%%%%
\section{Introduction}
A Gauss word is a finite sequence of letters where each letter that
appears, appears exactly twice.
It is well known that any oriented virtual knot diagram has an
associated Gauss word (for example \cite{Kauffman:VirtualKnotTheory}).  
We can construct a Gauss word from a diagram in the following way. We
label the real crossings of the diagram and pick an arbitrary base point
on the curve away from any real crossings. Starting from the base point,
we follow the curve according to its orientation, noting the labels of
the real crossings we pass through. We stop when we get back to the base
point. As we passed through each real crossing twice, the resulting
sequence of letters is a Gauss word.
\par
A Gauss phrase is a tuple (ordered list) of sequences of letters where
concatenating all the sequences in the tuple together gives a Gauss
word.
By assigning an order to the components of a virtual link diagram,
we can construct a Gauss phrase from the diagram in a similar way to the
method used for a knot diagram.
In this case, we get a sequence of letters for each component of the
virtual link diagram.
By analogy, a single sequence of letters in the tuple is called a
component of the Gauss phrase. 
A Gauss phrase with just one component is necessarily a Gauss word.
\par
Note that when we construct a Gauss phrase from a virtual link diagram
we do not record the over and under crossing information from each real
crossing.
This means that even after flattening a virtual link diagram by
flattening each real crossing to a double point, we are still able to
construct the same Gauss phrase.
We call such diagrams \emph{flattened virtual link diagrams}.
\par
Two virtual link diagrams represent the same virtual link if they are
related by a finite sequence of diagrammatic moves known collectively as
generalized Reidemeister moves (these moves
are described, for example, in \cite{Kauffman:VirtualKnotTheory}).
For simplicity, in this paper we assume that the components of each
virtual link are assigned an order and this order is preserved under
the diagrammatic moves.
By flattening the generalized Reidemeister moves, we can
define an equivalence relation on flattened virtual link diagrams.
The equivalence classes of this relation are (multi-component) virtual
strings (also known as flattened virtual links).
Single-component virtual strings have been studied in various papers,
for example in \cite{Turaev:2004}, \cite{Gibson:tabulating-vs} and
\cite{Gibson:ccc}.
Each virtual link has an associated virtual string (derived by
flattening a diagram of the virtual link), which, by definition of the
moves on virtual strings, is an invariant of the virtual link.
\par
By studying how Gauss phrases change under the flattened generalized
Reidemeister moves, we can define combinatorial moves on Gauss phrases
and say that two Gauss phrases are equivalent if they are related by a
finite sequence of these moves. 
We call this relation \emph{homotopy}.
Then, by definition, the homotopy equivalence class of a Gauss phrase
derived from a flattened virtual link diagram is an invariant of the
virtual string the diagram represents.
We thus have a hierarchy of objects.
Each virtual link has an associated virtual string and each virtual
string has an associated Gauss phrase homotopy class.
\par
Suppose we are given two virtual strings.
If we can show that their associated Gauss phrase homotopy classes are
different, we can conclude that the two virtual strings are not
equivalent. 
Similarly, given two virtual knots, showing non-equivalence of the
associated Gauss phrase homotopy classes proves the two virtual knots to
be distinct. 
This gives some motivation for studying homotopy of Gauss phrases.
\par
On the other hand, as we lose information in the transition from virtual
links to Gauss phrases, the homotopy classes of a Gauss phrase is
clearly not a complete invariant of a virtual link.
Indeed, a Gauss phrase derived from any classical link with
$n$-components is homotopic to the $n$-component Gauss phrase where each
component is an empty word.
In fact, it is possible, for example using Turaev's $u$-polynomial
defined in \cite{Turaev:2004}, to show that the map from virtual strings
to homotopy classes of Gauss phrases is not a bijection. 
\par
One of the combinatorial moves we define on Gauss phrases is called the
\emph{shift move}.
It allows us to move the first letter of any component to the end of the
same component.
Diagrammatically, this corresponds to pushing a base point along a
curve, through a real crossing.
By disallowing this move, we can define a new kind of homotopy which we
call \emph{open homotopy}.
\par
Homotopy of Gauss phrases is in fact just a type of nanophrase homotopy.
Nanophrases were introduced by Turaev in \cite{Turaev:KnotsAndWords}.
A nanophrase is a Gauss phrase with a map from the alphabet of the Gauss
phrase to some fixed set $\alpha$. This map is called a projection. 
A nanoword is a nanophrase for which the Gauss phrase has just a single
component.
\par
Homotopy on nanophrases is also defined by moves on the associated Gauss
phrases. However, whether and how moves can be applied is dependent on
the projection and on $\alpha$.
In this sense, homotopy of Gauss phrases can be viewed as homotopy of
nanophrases with $\alpha$ just containing a single element.
One of the moves on nanophrases corresponds to the shift move for Gauss
phrases. 
If we disallow this move, we can define open homotopy for nanophrases.
\par
In \cite{Turaev:Words}, Turaev defined various open homotopy invariants of
nanowords (in that paper Turaev uses the term homotopy where we use open
homotopy).
All these invariants are trivial in the case where $\alpha$ contains a
single element. 
This fact led Turaev to conjecture (in \cite{Turaev:Words}) that
nanowords with $\alpha$ containing a single element are all
homotopically equivalent to the empty word.
In other words, the conjecture states that there is just a single
homotopy equivalence class of Gauss words. 
All the invariants that we consider in this paper are trivial for Gauss
words and so can not be used to provide a counter-example to Turaev's
conjecture.
\par
In \cite{Fukunaga:nanophrases} and \cite{Fukunaga:nanophrases2} Fukunaga
studies open homotopy of nanophrases.
In this paper we will see that some invariants he defines for open
homotopy of Gauss words are in fact homotopy invariants.
One such invariant can be defined as a symmetric $n \times n$ matrix
with elements in $\cyclic{2}$, where $n$ is the number of components in
the Gauss phrase.
In this paper we call this invariant the linking matrix.
Fukunaga also defines an open homotopy invariant of Gauss phrases called
$T$ which is not a homotopy invariant.
\par
In this paper we define a new open homotopy invariant of Gauss phrases
called $S_o$.
For an $n$-component Gauss phrase $p$, $S_o(p)$ is an $n$-tuple of
subsets of $\cyclicproduct{2}{n}$, one subset for each component.
Fukunaga's $T$ invariant can be calculated from $S_o$ but $S_o$ is
stronger than $T$.
On the other hand, $S_o$ and the linking matrix are independent.
\par
Using a similar construction we define a homotopy invariant called $S$.
For an $n$-component Gauss phrase, this invariant takes a value which is
an $n$-tuple of pairs, where each pair consists of a vector in
$\cyclicproduct{2}{n}$, taken from the linking matrix, and a subset of a
certain quotient of $\cyclicproduct{2}{n}$ by the vector.
In fact, the $S$ invariant can be calculated from the $S_o$ invariant
and the linking matrix.
However, as an open homotopy invariant, $S$ is weaker than a combination
of $S_o$ and the linking matrix.
\par
For both $S_o$ and $S$ we provide a canonical way of writing the
invariants as an $n$-tuple of matrices.
This gives a convenient way to record the invariants and makes checking
equivalence easier.
\par
We determine the values that the invariants $S_o$ and
$S$ can attain and in both cases we give a method to construct a Gauss
phrase with a given attainable value of the invariant.
In particular, we observe that for each $n$, the number of values that
$S_o$ or $S$ can attain over the entire set of $n$-component Gauss
phrases is finite. 
\par
We organise the rest of this paper as follows.
In Section~\ref{sec:gaussphrases} we give a formal definition of Gauss
phrases, open homotopy and homotopy.
In Section~\ref{sec:simple_invariants} we describe some homotopy
invariants of Gauss phrases which follow easily from the definition of
homotopy. 
\par
In Section~\ref{sec:so_invariant} we describe Fukunaga's $T$ invariant
and use it to show that homotopy and open homotopy are not
equivalent.
We then define the open homotopy invariant $S_o$ and show it to be
stronger than the $T$ invariant.
In Section~\ref{sec:so_conditions} we consider what values $S_o$ may
attain.
\par
We turn our attention back to homotopy in Section~\ref{sec:s_invariant}
where we define the $S$ invariant and prove its invariance. 
We study the invariant further in Section~\ref{sec:s_conditions} where
we consider what values $S$ may attain. 
\par
Finally, in Section~\ref{sec:other_homotopy}, we briefly discuss another
type of Gauss phrase homotopy where we allow permutations of
components.
We also suggest the possibility of defining stronger invariants based on
$S_o$ or $S$ for other kinds of nanophrase homotopy. 
\begin{ack}
The author would like to express his thanks to his supervisor Hitoshi
 Murakami for all his help, advice and encouragement.
He is also grateful to Vladimir Turaev for comments on an earlier
 version of this paper which led to several improvements, and for
 suggesting the proof given for Lemma~\ref{lem:odd_parity} which is much
 simpler than the author's original proof.
\end{ack}
%%%%%%%%%%%%%%%%%%%%%%%%%%%%%%%%%%%%%%%%%%%%%%%%
\section{Gauss phrases}\label{sec:gaussphrases}
An \emph{alphabet} is a finite set and elements of an alphabet are
called \emph{letters}. A \emph{word} on an alphabet $\mathcal{A}$ is a
finite sequence of letters of $\mathcal{A}$. The \emph{length} of a word
is the number of letters appearing in the sequence. 
Formally we can define a word as a map from the set $\{1, \dotsc, n\}$ to
$\mathcal{A}$ where $n$ is the length of the word.
However, we usually write the word as a sequence of unseparated letters
from which the map, if needed, can be deduced in the obvious way.
For example the words $A$, $BCB$ and
$CAABBCACBA$ are all words on the alphabet ${A,B,C}$. 
The empty sequence is a word on any alphabet. It has length $0$ and is
written $\trivial$.
\par
The concatenation of two words, $x$ and $y$, on
an alphabet $\mathcal{A}$ is defined to be the sequence of letters
constructed by joining the sequence of letters in $y$ to the end of the
sequence of letters in $x$. We write $xy$ for the concatenation of $x$
and $y$. For example, let $x$ be $A$ and $y$ be $BCB$, then $xy$ is
$ABCB$ and $yx$ is $BCBA$. It is clear that concatenation is
associative, that is $(xy)z$ is equal to $x(yz)$ for arbitrary words
$x$, $y$ and $z$. Thus we can unambiguously write $x_1x_2\dotsm x_n$ for
the concatenation of $n$ words $x_1, \dotsc, x_n$.
\par
A \emph{Gauss word} on an alphabet $\mathcal{A}$ is a word $w$ on
$\mathcal{A}$ such that every letter of $\mathcal{A}$ appears in $w$
exactly twice. For example, $ABCABC$ is a Gauss word on ${A,B,C}$. On
the other hand $ABAB$ is not a Gauss word on ${A,B,C}$ because $C$ does
not appear in the word. However $ABAB$ is a Gauss word on ${A,B}$.
A Gauss word on an alphabet $\mathcal{A}$ necessarily
has length $2n$ where $n$ is the number of letters in $\mathcal{A}$.
\par
A \emph{Gauss phrase} on an alphabet $\mathcal{A}$ is a sequence of
words $x_1, \dotsc, x_m$ on $\mathcal{A}$ such that $x_1x_2\dotsm x_m$
is a Gauss word on $\mathcal{A}$. We call the $x_i$ the components of
the Gauss phrase. Note that when a Gauss phrase has only one component,
that component must be a Gauss word.
\par
An isomorphism of Gauss phrases is defined as follows. Let
$p_1$ be a Gauss phrase on $\mathcal{A}_1$ and $p_2$ be a Gauss phrase on
$\mathcal{A}_2$.
If $f$ is a bijection mapping $\mathcal{A}_1$ onto
$\mathcal{A}_2$ such that $f$ applied letterwise to $p_1$ gives $p_2$,
then $f$ is an isomorphism of Gauss phrases. We say that
$p_1$ and $p_2$ are isomorphic if such an isomorphism
between them exists. 
\par
We now define some moves on Gauss phrases.
When describing moves on Gauss phrases we use the following
conventions. Upper case letters $A$, $B$ and $C$ represent arbitrary
individual
letters of a Gauss phrase. Lower case letters $w$, $x$, $y$, $z$ represent
sequences of letters, possibly including the $|$ character which
separates components of the Gauss phrase. Of course, in each move we
start and end with a Gauss phrase.
\par
The move H1 allows us to transform a Gauss phrase of the form $xAAy$
into the form $xy$. In other words, if the two
occurences of a letter appear adjacent in the same component, they can
be removed. The inverse of the move allows us to add some letter $A$ not
already appearing in the Gauss phrase. We can insert the subword $AA$
into any one component of the Gauss phrase. 
\par
The move H2 allows us to take a Gauss phrase of the form $xAByBAz$ and
transform it to the form $xyz$. Here, the subword $AB$ must
appear in a single component. The subword $BA$ must also appear in a
single component, possibly, but not necessarily, the same component as
the subword $AB$. The inverse of the move allows us to add some pair of
letters $A$ and $B$ not already appearing in the Gauss phrase.
\par
The move H3 allows us to transform a Gauss phrase of the form $wABxACyBCz$
into the form $wBAxCAyCBz$. Here each subword ($AB$, $AC$ and $BC$)
appears as a subword of a component in the Gauss phrase. The subwords
may all appear in the same component, or some may appear in different
components. The inverse move allows us to reverse the operation, that is
any Gauss phrase of the form $wBAxCAyCBz$ can be transformed to
$wABxACyBCz$.
\par
Any two Gauss phrases are said to be \emph{open homotopic} if there
exists a finite sequence of isomorphisms, and the moves H1, H2 and H3
and their inverses, which transform one Gauss phrase into the other.
It is easy to see that this relation is an equivalence relation.
We call this equivalence relation \emph{open homotopy} and collectively
the moves H1, H2, H3 and their inverses are called \emph{homotopy
moves}. 
\par
We also introduce a move called the \emph{shift move} which allows us to
take the first letter in any component of a Gauss phrase and move it to
the end of the same component.
Suppose we have a Gauss phrase $p$ where the $k$th component is $Ay$ for
some arbitrary sequence of letters $y$.
Applying the shift move to the $k$th component of $p$ gives a new Gauss
phrase $q$. 
The $k$th component of $q$ is $yA$ and for any $i$ not equal to $k$, the
$i$th component of $q$ is the same as the $i$th component of $p$.
\par
Any two Gauss phrases are said to be \emph{homotopic} if there exists a
finite sequence of isomorphisms, shift moves and the moves H1, H2 and H3
and their inverses, which transform one Gauss phrase into the other.
This relation is an equivalence relation which we call \emph{homotopy}.
\par
Note that homotopy is just open homotopy modulo the shift move.
Therefore, if two Gauss phrases are open homotopic, they are also
homotopic.
This means there is a well defined map from the set of equivalence
classes under open homotopy to the set of equivalence classes under
homotopy.
It also means that any homotopy invariant of Gauss phrases is also
invariant under open homotopy. 
On the other hand, the fact that two Gauss phrases are homotopic does not
necessarily imply that they are open homotopic.
An example to show this is given in Example~\ref{ex:closed-not-open}.
\par
We note that the homotopy moves and shift move were originally
introduced by Turaev in the more general case of nanowords
(\cite{Turaev:Words}) and nanophrases (\cite{Turaev:KnotsAndWords}).
Indeed, as mentioned in the introduction, homotopy of Gauss phrases is
just a type of homotopy of nanophrases.
%%%%%%%%%%%%%%%%%%%%%%%%%%%%%%%%%%%%%%%%%%%%%%%%
\section{Simple invariants}\label{sec:simple_invariants}
We start by noting some very simple homotopy invariants of Gauss phrases.
\par
As none of the moves allow us to add or remove components, the number of
components of a Gauss phrase is invariant.
\par
We define $\complength(i)$ to be the number of letters in the $i$th component of
the Gauss phrase modulo $2$. We have the following proposition. 
\par
\begin{prop}
For each $i$, $\complength(i)$ is a homotopy invariant.
\end{prop}
\begin{proof}
As isomorphism, the shift move and move H3 does not change the number of
 letters in any component, it is clear that $\complength(i)$ is invariant under
 these operations. 
Move H1 only affects one component, either adding or removing two
 letters. However, modulo $2$, the number of letters is unchanged.
Move H2 affects either one or two components. 
In each case letters are added or removed in pairs, so again, modulo
 $2$, the number of letters in any component is unchanged.
Thus $\complength(i)$ is invariant under homotopy.
\end{proof}
\begin{rem}
In Lemma~4.5 of \cite{Fukunaga:nanophrases}, Fukunaga makes the same
 observation for open homotopy of nanophrases.
\end{rem}
\par
In particular we can define the \emph{component length vector}
$(\complength(1),\dotso,\complength(n))$ in $\cyclicproduct{2}{n}$.
This vector is a homotopy
invariant. Note that since the total number of 
letters appearing in a Gauss phrase is even, the sum of $\complength(i)$ over all
$i$ must be $0$ modulo $2$. This gives a necessary condition on a vector
in $\cyclicproduct{2}{n}$ to be component length vector for some
Gauss phrase. It is easy to see that this condition is also sufficient.
\par
\begin{ex}
The Gauss phrase $\trivial|\trivial$ has component length
 vector $(0,0)$ and the Gauss phrase $A|A$ has component length vector
 $(1,1)$. Thus these two Gauss phrases are distinct.
\end{ex}
\par
As each letter in a Gauss phrase appears exactly twice, we can consider
which component or components each letter appears in.
If both occurences of a letter appear in the same component we call the
letter a \emph{single-component letter}.
If the two occurences of a letter appear in different components we call
the letter a \emph{two-component letter}.
Note that the move H1 can only be applied to single-component
letters.
Also note that the two letters involved in an H2 move are either
both single-component letters or both two-component letters.
\par
We now define the \emph{linking matrix} $L$ of an $n$-component
Gauss phrase. 
We define $L$ to be an $n \times n$ matrix and we denote the elements of
$L$ by $l_{ij}$ where $i$ and $j$ are integers between $1$ and $n$
inclusive.
We set $l_{ii}$ to be $0$ for all $i$. When $i$ does not equal $j$ we
set $l_{ij}$ to be the number of two-component letters appearing both in
the $i$th component and the $j$th component, modulo $2$. 
By definition, $l_{ij}$ is equal to $l_{ji}$ for all $i$ and $j$.
This means that $L$ is a symmetric matrix with elements in
$\cyclic{2}$.
\begin{thm}
The linking matrix is a homotopy invariant of a Gauss phrase.
\end{thm}
\begin{proof}
The shift move, isomorphism and move H3 do not add or remove letters
 from a Gauss phrase. They also do not change single-component letters
 into two-component letters or vice-versa. Thus the linking matrix is
 invariant under these operations.
\par
We have already noted that the move H1 can only be applied to
 single-component letters. As the status of other letters in the
 Gauss phrase are unaffected, the linking matrix is invariant under this
 move too.
\par
If the move H2 is applied to single-component letters, as for the H1
 move, the linking matrix is invariant.
Suppose the move H2 is applied to two two-component letters. The two letters must
 both appear in in the same pair of components, say component $i$ and
 component $j$ (where $i$ does not equal $j$). Then the number
 of two-component letters appearing both in the $i$th component and the
 $j$th component changes by $2$ under the H2 move. However, since we
 consider $l_{ij}$ modulo $2$, $l_{ij}$ does not change and so the
 linking matrix is invariant under this move. 
\end{proof}
\begin{rem}
After writing this section we discovered that in
 \cite{Fukunaga:nanophrases2}, Fukunaga has defined an equivalent
 invariant in the case of open homotopy.
In fact, in that paper Fukunaga defines the invariant for the more
 general case of any open homotopy of nanophrases.
Invariance is stated in Proposition~5.10 of
 \cite{Fukunaga:nanophrases2}. 
\end{rem}
We note that the component length vector of a Gauss phrase can be
calculated from the linking matrix. The $i$th element in the vector is
given by 
\begin{equation*}
\complength(i) = \sum_{j=1}^n l_{ij} \mod 2.
\end{equation*}
We now give an example to show that the linking matrix is a stronger
invariant than the component length vector.
\begin{ex}
Let $\alpha$ be the Gauss phrase $\trivial|\trivial|\trivial$ and $\beta$
 be the Gauss phrase $AB|AC|BC$. Then both $\alpha$ and $\beta$ have the 
 component length vector $(0,0,0)$. On the other hand, the linking
 matrix for $\alpha$ is
\begin{equation*}
\begin{pmatrix}
0 & 0 & 0 \\
0 & 0 & 0 \\
0 & 0 & 0
\end{pmatrix},
\end{equation*}
but the linking matrix for $\beta$ is
\begin{equation*}
\begin{pmatrix}
0 & 1 & 1 \\
1 & 0 & 1 \\
1 & 1 & 0
\end{pmatrix}.
\end{equation*}
Thus $\alpha$ and $\beta$ are distinct Gauss phrases under homotopy.
\end{ex}
Let $L$ be an $n \times n$ symmetric matrix with elements in
$\cyclic{2}$ such that the entries on the main diagonal are all zero.
Then we can easily construct a Gauss phrase that has linking matrix
$L$. 
\par
We use $l_{ij}$ to denote the elements of $L$.
We start with the $n$-component Gauss phrase that has all $n$
components empty.  Consider each pair $(i,j)$ where 
$1 \leq i < j \leq n$. If $l_{ij}$ is $1$ we add a letter $A_{ij}$ to
the $i$th and $j$th components. If $l_{ij}$ is $0$ we do nothing.
After we have considered all such pairs $(i,j)$, the result is a
Gauss phrase which has $L$ as its linking matrix.
\par
\begin{ex}
Let $L$ be the matrix
\begin{equation*}
\begin{pmatrix}
0 & 0 & 1 & 1 \\
0 & 0 & 1 & 0 \\
1 & 1 & 0 & 1 \\
1 & 0 & 1 & 0
\end{pmatrix}.
\end{equation*}
Then the Gauss phrase
$A_{13}A_{14}|A_{23}|A_{13}A_{23}A_{34}|A_{14}A_{34}$ has linking
matrix $L$.
\end{ex}
%%%%%%%%%%%%%%%%%%%%%%%%%%%%%%%%%%%%%%%%%%%%%%%%
\section{The invariant $S_o$}\label{sec:so_invariant}
In this section we focus on open homotopy of Gauss phrases.
We start by recalling the $T$ invariant defined by Fukunaga.
We use it to show that open homotopy is not equivalent to homotopy.
\par
The $T$ invariant is an open homotopy invariant of Gauss phrases.
The original definition was given in Definition~3.10 of
\cite{Fukunaga:nanophrases}.
We give an equivalent definition of the invariant here.
Let $p$ be an $n$-component Gauss phrase.
We say that a single-component letter $A$ has odd parity if the number
of letters appearing between the two occurences of $A$ is odd. We then
define $T_k(p)$ to be the number of odd parity single-component letters
appearing in the $k$th component, modulo $2$. Thus $T_k(p)$ is in
$\cyclic{2}$.
We then define $T(p)$ to be the vector in $\cyclicproduct{2}{n}$ where the
$k$th element is $T_k(p)$ for each $k$.
It is easy to check that $T(p)$ is an open homotopy invariant and that
it is equivalent to Fukunaga's $T$ invariant.
\begin{ex}\label{ex:closed-not-open}
Let $p$ be the Gauss phrase $A|A$. Then $T(p)$ is the vector $(0,0)$.
\par
Let $q$ be the Gauss phrase $ABA|B$. Then $T(q)$ is the vector
 $(1,0)$. This means that $q$ is not open homotopic to $p$.
\par
Under homotopy we can apply a shift move to $q$ to get the Gauss
 phrase $BAA|B$. Applying an H1 move to this Gauss phrase gives us $B|B$
 which is isomorphic to the Gauss phrase $p$. 
Thus $q$ is homotopic to $p$.
\end{ex}
This example shows that homotopy and open homotopy are not equivalent.
It also shows that Fukunaga's $T$ invariant is not invariant under
homotopy.
\par
In this section we define a new invariant $S_o$ for open homotopy of
Gauss phrases. 
We now introduce some new terminology.
\par
We write $K_n$ for $\cyclicproduct{2}{n}$ and consider elements of $K_n$
as vectors.
\par
Let $p$ be an $n$-component Gauss phrase. Let $w$ be a subword of a
component of $p$. We define the \emph{linking vector} of $w$ to be a
vector $\cvector{v}$ in $K_n$. The $i$th element of
$\cvector{v}$ is the number of letters, modulo $2$, which occur just once
in $w$ and for which the other occurence of the letter appears in the
$i$th component of $p$. 
Note that the order of the letters in $w$ is not important, only which
letters appear.
We remark that if $w$ is the whole of $k$th component
of $p$ then the linking vector of $w$ is the $k$th row (or column) of
the linking matrix for $p$.
We write the linking vector of $w$ as $\gvector{l(w)}$.
\par
We define the \emph{linking vector} of a single-component letter $A$ as
follows. Suppose $A$ is in the $k$th component of $p$. Then that
component has the form $xAyAz$ for some, possibly empty, words $x$, $y$
and $z$. We define the linking vector of $A$ to be the linking vector of
the word $y$. We write it $\gvector{l(A)}$.
\begin{ex}
Consider the Gauss phrase $ABAC|DBEDFEG|CFG$. Then the linking vector of $A$
 is $(0,1,0)$, the linking vector of $D$ is $(1,1,0)$ and the linking
 vector of $E$ is $(0,1,1)$.
\end{ex}
Write $\mathcal{A}_k$ for the set of single-component letters appearing
in the $k$th component of a Gauss phrase $p$.
For a vector $\cvector{v}$ in $K_n$ we define a map $d_k$ from $K_n$ to
$\Z$ by 
\begin{equation*}
 d_k(\cvector{v}) = \sharp \lbrace X\in \mathcal{A}_k | \gvector{l(X)} =
  \cvector{v} \rbrace 
\end{equation*}
where $\sharp$ means the number of elements in the set.
Let $B_k(p)$ be the subset of $K_n - \{\cvector{0}\}$ given by
\begin{equation*}
 B_k(p) = \lbrace \cvector{v} \in K_n - \{\cvector{0}\} |
  d_k(\cvector{v}) \text{ is odd}\rbrace.
\end{equation*} 
In other words $B_k(p)$ is the set of non-zero vectors of $K_n$ which are
the linking vectors of an odd number of single-component letters in the
$k$th component of $p$.
\par 
We define $S_o(p)$ to be the $n$-tuple where for each $k$ the $k$th element
is the set $B_k(p)$.
\begin{thm}\label{thm:so_invariant}
Let $p$ be a Gauss phrase. Then $S_o(p)$ is an open homotopy invariant
 of $p$.
\end{thm}
\begin{proof}
It is enough to prove that $B_k(p)$ is an open homotopy invariant of
 $p$.
We must show that $B_k(p)$ is unchanged for each homotopy move.
\par
An isomorphism of a Gauss phrase just changes the labels of the
 letters. This has no effect on the linking vectors of any
 single-component letters and thus has no effect on the set $B_k(p)$. This
 means that $B_k(p)$ is invariant under isomorphism. 
\par
We now note that any move which does not involve the $k$th component of
 $p$ can not affect $B_k(p)$. Thus we only need to consider moves involving
 the $k$th component.
\par
The move H1 takes a Gauss phrase $p$ with $k$th component $xAAy$ and
 gives a Gauss phrase $q$ with $k$th component $xy$. Consider any other
 single-component letter $B$ in the $k$th component. The letter $A$ only
 has an effect on the linking vector for $B$ if there is just one
 occurence of $A$ appearing between the two occurences of $B$ in the
 component. However, since the two occurences of $A$ must
 appear adjacently, this is not possible. Thus the presence or absence
 of the two occurences of $A$ in the component has no effect on the
 linking vector of $B$.
\par
The linking vector of $A$ itself is $\cvector{0}$. 
Since $B_k(p)$ is defined to be a subset of $K_n - \{\cvector{0}\}$,
 adding or removing $A$ from the Gauss phrase has no effect
 on $B_k(p)$. Thus $B_k(p)$ is invariant under the move H1 and its inverse.
 \par
The move H2 has two forms depending on whether the letters removed are
 single-component letters or two-component letters. 
In either case let $A$ and $B$ be the letters that are removed. Now
 consider any other 
 single-component letter $C$ in the $k$th component. 
When we calculate the linking vector for $C$, one or both of the
 subwords $AB$ and $BA$ may appear between the two occurences of
 $C$. However, this means that when we calculate the vector, we have to
 count $A$ if and only if we have to count $B$.
The other occurences of $A$ and $B$ appear in the same component and so
 $A$ and $B$ are counted equivalently. 
Since we count modulo two, their contributions cancel and
 so the linking vector for $C$ is the same whether $A$ and $B$ are both
 present or both removed.
\par
If $A$ and $B$ are two-component letters then neither $A$ nor $B$ are in
 $\mathcal{A}_k$. 
Therefore the removal or insertion of $A$ and $B$ has no effect on
 $B_k(p)$. 
\par
If $A$ and $B$ are single-component letters then the move takes a Gauss
 phrase $p$ with $k$th component $xAByBAz$ and gives a Gauss phrase $q$
 with $k$th component $xyz$. 
In $p$, the linking vectors for $A$ and $B$ are both given by
 $\gvector{l(y)}$. 
Thus, when we apply the H2 move or its inverse, $d_k(\gvector{l(y)})$
 changes by plus or minus two.
This means that $\gvector{l(y)}$ is in $B_k(q)$, if and only
 if $\gvector{l(y)}$ is in $B_k(p)$.
In other words, $B_k(p)$ is invariant under move H2 or its inverse.
\par
We now consider move H3 which takes a Gauss phrase of the form
 $wABxACyBCz$ and gives us one of the form $wBAxCAyCBz$. For any
 single-component letter $D$ uninvolved in the move, the linking vector
 for $D$ is unchanged by the move. This is because the linking vector is
 unaffected by changing the order of any letters between the two
 occurences of $D$.
\par
For a single-component letter $X$ involved in the H3 move, we first note
 that the other two letters involved in the move must both be of the
 same type for the purposes of calculating the linking vector of
 $X$.
More precisely, there are two possible cases.
The first case is that both the other letters are single-component
 letters.
The second case is that both the other letters are two-component
 letters.
In this case, the other occurences of the two letters must appear in the
 same component.
\par
Assume $A$ is a single-component letter. Then the difference between the
 linking vectors for $A$ before and after the moves is that $B$ is 
 counted in the linking vector before the move whereas $C$ is counted in
 the linking vector after the move. However, as $B$ and $C$ are counted
 equivalently it means the linking vector for $A$ is unchanged.
\par
If $B$ is a single-component letter then $A$ and $C$ are necessarily
 single-component letters too. We consider the difference between the
 linking vectors for $B$ before and after the move.
 Before the move, $A$ and $C$ are both counted. As they are counted
 equivalently and we count modulo two, their contributions cancel. After the
 move, neither $A$ nor $C$ are counted. Thus the linking vector for $B$
 is unchanged by the move.
\par
The case where $C$ is a single-component letter is symmetrical to the
 case for $A$. Here $B$ is counted in the linking vector before the move
 and $A$ is counted in the linking vector after the move, meaning that
 overall the linking vector is unchanged. Thus the linking vector for
 $C$ is unchanged by the move.
\par
Thus we can conclude that $B_k(p)$ is unchanged by the move H3 or its
 inverse which completes the proof.
\end{proof}
To make it easy to compare the $S_o$ invariant for two different Gauss
phrases, we provide a canonical way of writing each $B_k(p)$ as a
matrix. We can then write $S_o(p)$ as a tuple of matrices. 
Given two Gauss phrases $p$ and $q$, determining whether $S_o(p)$ is
equivalent to $S_o(q)$ 
or not becomes a simple matter of determining whether the two tuples of
matrices are equal or not.
\par
We start by defining an order on the vectors of $K_n$ in the following
way.
Let $\cvector{u}$ and $\cvector{v}$ be two vectors in $K_n$. 
Let $u_i$ be the $i$th element in $\cvector{u}$ and $v_i$ be the $i$th
element in $\cvector{v}$. 
Then $\cvector{u}$ is less than $\cvector{v}$ if there exists an $i$
such that $u_i$ is $0$, $v_i$ is $1$ and for all $j$ less than $i$,
$u_j$ is equal to $v_j$.
\par
Suppose $p$ is an $n$-component Gauss phrase and $B_k(p)$ has $r$
elements.
We define a matrix $M_k$ for $B_k(p)$ as follows.
If $r$ is $0$, $M_k$ is the zero matrix with $n$ columns and $1$ row.
Otherwise, $M_k$ is the matrix with $n$ columns and $r$ rows where the
rows of $M_k$ are the elements of $B_k(p)$ written out in ascending
order.
\par
It is clear that given such a matrix we can reconstruct the set
$B_k(p)$.
On the other hand, there is no room for choice in the construction of
the matrix and so, under this construction method, there is only one
matrix which represents $B_k(p)$.
Thus for two Gauss phrases $p$ and $q$, $B_k(p)$ is equivalent to
$B_k(q)$ if and only if their corresponding matrices are 
equal.
\par
Note that we can calculate Fukunaga's $T$ invariant from $S_o$.
Recall that for a Gauss phrase $p$, $T_k(p)$ is, modulo two, the number
of odd parity single-component letters appearing in the $k$th component.
We say that a vector in $K_n$ is \emph{odd} if the sum of its elements is
$1$ modulo two.
It is simple to verify that $T_k(p)$ is equal, modulo two, to the
number of odd vectors in $B_k(p)$.
\par
The following example shows that the invariant $S_o$ is stronger than
the invariant $T$.
\begin{ex}
Consider the Gauss phrase $p$ given by $ACBADBEF|CE|DF$.
Then $T(p)$ is $(0,0,0)$ which means $T$ can not distinguish $p$ from
 the trivial $3$-component Gauss phrase
$\trivial | \trivial | \trivial$. 
\par
On the other hand $S_o(p)$ is given by the $3$-tuple of matrices
\begin{equation*}
\begin{pmatrix}
\begin{pmatrix}
1 & 0 & 1 \\
1 & 1 & 0 \\
\end{pmatrix}, 
\begin{pmatrix}
0 & 0 & 0 \\
\end{pmatrix}, 
\begin{pmatrix}
0 & 0 & 0 \\
\end{pmatrix}
\end{pmatrix}.
\end{equation*}
This shows that $p$ is not equivalent to the trivial $3$-component Gauss
 phrase under open homotopy.
\end{ex}
%%%%%%%%%%%%%%%%%%%%%%%%%%%%%%%%%%%%%%%%%%%%%%%%
\section{Properties and realizability of $S_o$}\label{sec:so_conditions}
In this section we determine what possible values $S_o$ can attain. We
start by asking the following question.
Suppose $B$ is a subset of $K_n - \{\cvector{0}\}$.
What properties must $B$ satisfy for $B$ to be $B_k(p)$ 
of the $k$th component of some Gauss phrase $p$?
\par
In order to answer this question we must first make some definitions.
We say that a vector in $K_n$ is \emph{$k$-odd} if the $k$th element of
the vector is $1$ and say that it is \emph{$k$-even} if the $k$th
element is $0$. 
A necessary condition on $B$ is given in the following proposition.
\begin{prop}\label{prop:so_necessary}
Let $p$ be a Gauss phrase.
Then $B_k(p)$ contains an even number of $k$-odd vectors.
\end{prop}
\begin{proof}
Let $D(k)$ be the set of single-component letters in the $k$th component
 of $p$ for which the linking vector is $k$-odd.
Then the number of $k$-odd vectors in $B_k(p)$ is equal, modulo two, to
 the number of letters in $D(k)$.
Thus we just need to show that the number of letters in $D(k)$ must be
 even.
\par
Note that if we take just the $k$th component of $p$
 and delete all the two-component letters from it, we still have enough
 information to calculate whether the linking vector for each
 single-component letter is $k$-odd or not.
As the resulting word just contains the single-component letters, it is
 a Gauss word. We call this Gauss word $w_k$.
For any letter $A$ in a Gauss word, the Gauss word has the form
 $xAyAz$. 
We say that $A$ has odd parity if $y$ is an odd length word and
 that $A$ has even parity if $y$ is an even length word.
Then a letter is in $D(k)$ if and only if it has odd parity in $w_k$.
The proof of the proposition is then complete if we can prove that the
 number of odd parity letters in $w_k$ is even. This fact is proven in
 Lemma~\ref{lem:odd_parity}.
\end{proof}
\begin{lem}\label{lem:odd_parity}
The number of odd parity letters in any Gauss word is even.
\end{lem}
\begin{proof}
Let $w$ be a Gauss word.
We say that two letters $A$ and $B$ in $w$ are \emph{linked} if $w$ has
 the form $uAvBxAyBz$ or $uBvAxByAz$. 
In other words, $A$ and $B$ are linked if they appear alternating in
 $w$.
Define $l(A,B)$ to be $1$, if $A$ and $B$ are different letters which
 are linked, and $0$ otherwise.
Note that because being linked is a symmetrical relationship, $l(A,B)$
 is equal to $l(B,A)$ for all letters $A$ and $B$ in $w$.
\par
It is easy to check that $A$ is odd if and only if the number of letters
 it is linked with is odd.
That is $A$ is odd if and only if
\begin{equation*}
\sum_{B\in w}l(A,B)
\end{equation*} 
is odd.
\par
Thus the number of odd letters in $w$ must have the same parity as
\begin{equation*}
\sum_{A\in w}\sum_{B\in w}l(A,B).
\end{equation*} 
By the symmetry of $l$, it is clear that this sum is even.
Thus the number of odd letters in $w$ must be even.
\end{proof}
We next show that this condition on a set $B$ is sufficient. We
have the following proposition.
\begin{prop}\label{prop:bk_construction}
Let $k$ and $n$ be positive integers such that $k$ is less than or equal
 to $n$.
Suppose that $B$ is a subset of $K_n - \{\cvector{0}\}$ which contains
 an even number of $k$-odd vectors.
Then there exists an $n$-component Gauss phrase $p$ for which $B_k(p)$
 is $B$ and every component other than the $k$th component contains no
 single-component letters.
\end{prop}
\begin{proof}
Suppose $B$ is empty.
We let $p$ be the $n$-component Gauss phrase for which every component
 is empty.
As $p$ contains no letters it is clear that $p$ is the required Gauss
 phrase.
\par
We now assume that $B$ is not empty.
Let $E(k)$ be the set of $k$-even vectors in $B$.
We define $r$ to be the number of elements in $E(k)$.
Let $O(k)$ be the set of $k$-odd vectors in $B$.
By the condition in the statement of the proposition, the number of
 elements in $O(k)$ is even. 
Thus we can say that the number of elements is $2s$ for some
 non-negative integer $s$. 
We arbitrarily arrange the elements of $O(k)$ into $s$ pairs.
\par
We now start the construction of $p$.
We begin with an $n$-component Gauss phrase for which every component is
 empty. 
We will define a word $x_i$ for each of the $r$ elements in $E(k)$ and
 define a word $y_i$ for each of the $s$ pairs in $O(k)$.
We will concatenate these words to make the $k$th component of $p$.
As we proceed in our construction we append letters to the other
 components of $p$ when necessary.
\par
For any vector $\cvector{v}$ in $K_n$ we define a \emph{linking subword}
 for the vector in the following way. 
We start with an empty word $u$. 
Looking at the vector $\cvector{v}$, we consider each element, except
 for the $k$th element, in turn. If the $j$th element is $1$, we append
 a previously unused letter to $u$ and append the same letter to the
 $j$th component of $p$. 
If the $j$th element is $0$, we do nothing.
Once we have finished, the word $u$ has the same number of letters as
 the number of non-zero elements in the vector $\cvector{v}$, excluding
 the $k$th element. 
This word $u$ is a linking subword for the vector.
\par
For a vector $\cvector{v}$ in $E(k)$ we define $x_i$ to be the word
 $X_iuX_i$ where $u$ is a linking subword for $\cvector{v}$ and $X_i$ is
 a previously unused letter.
\par
For a pair of vectors $\svector{v}{1}$ and $\svector{v}{2}$ in $O(k)$ we
 define $y_i$ to be the word $Y_iu_1Z_iY_iu_2Z_i$. Here $u_1$ is a linking
 subword for $\svector{v}{1}$ and $u_2$ is a linking subword for
 $\svector{v}{2}$. 
The letters $Y_i$ and $Z_i$ are previously unused letters.
\par
Let $w$ be the word given by concatenating all the words $x_i$ and all
 the words $y_i$.
The $k$th component of $p$ is then defined to be $w$.
\par
We now check that $B_k(p)$ is $B$.
Note that the $k$th component of the Gauss phrase has $r+2s$
 single-component letters, which we constructed in one-to-one
 correspondence with vectors in $B$.
Calculating the linking vectors for each single-component letter, it is
 easy to see that the linking vector is equal to the corresponding
 vector in $B$.
In particular, note that in the subwords $y_i$ the
 single-component letter $Z_i$ appears once between the two occurences
 of $Y_i$ and the single-component letter $Y_i$ appears once between the
 two occurences of $Z_i$. This means that in the linking vectors for
 $Y_i$ and $Z_i$ the $k$th element is $1$ as required. 
\par
By construction of $p$, single-component letters only appear in the $k$th
 component.
Thus $p$ is the required Gauss phrase. 
\end{proof}
\begin{ex}\label{ex:construct_single}
We construct a Gauss phrase $p$ for which $B_1(p)$ is represented by the
 matrix 
\begin{equation*}
\begin{pmatrix}
0 & 1 \\
1 & 0 \\
1 & 1 \\
\end{pmatrix}. 
\end{equation*}
\par
The matrix represents the set of vectors $\{(0,1),(1,0),(1,1)\}$.
Of these vectors, $(0,1)$ is $1$-even and the other two vectors are
 $1$-odd. So in this case $r$ and $s$ are both one and we just need
 to construct two words $x_1$ and $y_1$.
\par
Starting with an empty second component, we first construct $x_1$ from
 the vector $(0,1)$. As the second element in the vector is $1$, we
 define $x_1$ to be $X_1AX_1$. We append $A$ to the second component
 making it now $A$. 
\par
We now construct $y_1$ which has the form $Y_1uZ_1Y_1vZ_1$, where $u$ is
 defined by the contents of the vector $(1,0)$ and $v$ is defined by the
 contents of the vector $(1,1)$.
In each case we only need to consider the second element of the vector. 
For $(1,0)$, the second element is $0$, thus $u$ is the empty word and
 the word for the second component is left unchanged. 
For $(1,1)$, the second element is $1$. So we define $v$ to be $B$ and
 append a $B$ to the word for the second component. Thus $y_1$ is
 $Y_1Z_1Y_1BZ_1$ and the second component is now $AB$.
\par
Putting everything together, we get the Gauss phrase
 $X_1AX_1Y_1Z_1Y_1BZ_1|AB$.
\end{ex}
Suppose we are given an $n$-tuple of sets $B_i$, where each $B_i$ is a
subset of $K_n - \{\cvector{0}\}$. 
The following proposition shows that the necessary condition on each
$B_i$ given in Proposition~\ref{prop:so_necessary} is also sufficient.
\begin{prop}\label{prop:so_sufficient}
Let $n$ be a positive integer.
For each $i$ from $1$ to $n$ let $B_i$ be subset of 
$K_n - \{\cvector{0}\}$ containing an even number of $k$-odd vectors.
Then there exists a Gauss phrase $p$ such that $B_k(p)$ is equal to
 $B_k$ for $k$ running from $1$ to $n$. 
\end{prop} 
\begin{proof}
By Proposition~\ref{prop:bk_construction} we can construct a Gauss
 phrase $p_i$ for each component $i$ such that $B_i(p_i)$ is equal to
 $B_i$ and all the single-component letters in $p_i$ appear in the $i$th
 component.
We use isomorphisms to ensure that no letter appears in more than one of
 the Gauss phrases $p_i$.
We denote the $j$th component of $p_i$ by $w_{ij}$.
\par
Now define $x_i$ to be the concatenation
 $w_{1i}w_{2i}\dotso{}w_{ni}$. 
This is just the concatenation of the $i$th components of all $n$ Gauss
 phrases $p_i$. 
Then define $p$ to be the Gauss phrase where the $i$th component is
 $x_i$.
\par
Note that by construction, all the single-component letters in $k$th
 component of $p$ only come from the word $w_{kk}$. 
This means that $B_k(p)$ can be calculated just by looking at $w_{kk}$.
Therefore $B_k(p)$ is equal to $B_k(p_k)$ which is equal to $B_k$ as
 required.
\end{proof}
\begin{ex}\label{ex:so_realization}
We construct a Gauss phrase $p$ for which $S_o(p)$ is represented by the
 $2$-tuple of matrices 
\begin{equation*}
\begin{pmatrix}
\begin{pmatrix}
0 & 1 \\
1 & 0 \\
1 & 1 \\
\end{pmatrix}, 
\begin{pmatrix}
1 & 0 \\
\end{pmatrix}
\end{pmatrix}.
\end{equation*}
\par
We start by constructing a Gauss phrase for each matrix.
By Example~\ref{ex:construct_single}, the Gauss phrase for the first
 matrix is $X_1AX_1Y_1Z_1Y_1BZ_1|AB$.
\par
The second matrix represents the set containing a single vector
 $(1,0)$.
The corresponding word for that vector is $X_1AX_1$ and we set the first
 component to be $A$. 
This gives us a Gauss phrase which has the form $A|X_1AX_1$.
\par
As the Gauss phrases constructed from the two matrices have letters in
 common, we use isomorphisms to ensure all the letters are distinct. 
The Gauss phrase constructed from the first matrix becomes $XAXYZYBZ|AB$
 and the Gauss phrase from the second becomes $C|WCW$.
The Gauss phrase $p$ is then given by $XAXYZYBZC|ABWCW$.
\end{ex}
The following proposition shows that the linking matrix invariant and
the invariant $S_o$ are independent.
\begin{prop}\label{prop:so_linking}
Let $n$ be a positive integer.
For each $i$ from $1$ to $n$ let $B_i$ be subset of 
$K_n - \{\cvector{0}\}$ containing an even number of $k$-odd vectors.
Let $L$ be an $n \times n$ symmetric matrix with elements in
 $\cyclic{2}$ such that the entries on the main diagonal are all zero.
Then there exists a Gauss phrase $p$ such that $B_k(p)$ is equal to
 $B_k$ for $k$ running from $1$ to $n$ and the linking matrix of $p$ is
 $L$. 
\end{prop}
\begin{proof}
By Proposition~\ref{prop:so_sufficient} we can construct a Gauss phrase
 $q$ for which $B_k(q)$ is equal to $B_k$ for all $k$.
Let $M$ be the linking matrix for $q$.
We denote the elements of $L$ by $l_{ij}$ and the elements of $M$ by
 $m_{ij}$ where $i$ and $j$ are integers between $1$ and $n$ inclusive. 
\par
We now consider all possible pairs $i$ and $j$ such that $i$ is less
 than $j$.
If $l_{ij}$ and $m_{ij}$ are not equal, we append a previously unused
 letter to the $i$th and $j$th components of $q$.
Note that doing this does not change $B_k(q)$ for any component $k$. 
If $l_{ij}$ and $m_{ij}$ are equal, we do nothing.
Once we have considered all possible pairs, we call the resulting Gauss
 phrase $p$.
\par
As all the letters we added to $q$ do not change $B_k(q)$ for any $k$, 
$B_k(p)$ is equal to $B_k$ for all $k$.
Note that appending a two-component letter to the $i$th and $j$th
 components of a Gauss phrase swaps the corresponding entry in the
 linking matrix.
As we added one such letter for each pair ($i$,$j$) if and only if
 $l_{ij}$ and $m_{ij}$ were not equal, it is clear that the linking
 matrix of $p$ is $L$.  
\end{proof}
\begin{ex}
In Example~\ref{ex:so_realization} we constructed the Gauss phrase $p$
 given by $XAXYZYBZC|ABWCW$ which has linking matrix
\begin{equation*}
\begin{pmatrix}
0 & 1 \\
1 & 0 \\
\end{pmatrix}. 
\end{equation*}
To construct a Gauss phrase $q$ such that $S_o(q)$ equals $S_o(p)$ but
 has a zero linking matrix, we just need to append an unused letter to
 both components.
The result is $XAXYZYBZCD|ABWCWD$.
\end{ex}
\begin{rem}
Over the set of $n$-component Gauss phrases, the number of values that
 the invariant $S_o$ can take is finite.
This is because the number of vectors in $K_n$ is finite and so the
 number of subsets of $K_n - \{\cvector{0}\}$ is also finite.
\end{rem}
%%%%%%%%%%%%%%%%%%%%%%%%%%%%%%%%%%%%%%%%%%%%%%%%
\section{The invariant $S$}\label{sec:s_invariant}
In this section we construct an invariant $S$ for homotopy of Gauss
phrases based on the ideas in Section~\ref{sec:so_invariant}.
\par
We start by giving an example to show that the $S_o$ invariant is not
invariant under the shift move.
\begin{ex}\label{ex:so_strong}
Let $p$ be the Gauss phrase $ABAC|B|C$ and $q$ be the Gauss phrase
 $BACA|B|C$.
We can transform $p$ into $q$ by a shift move on the first component, so
 $p$ and $q$ are homotopic.
\par
On the other hand, $B_1(p)$ is $\{0,1,0\}$ and $B_1(q)$ is $\{0,0,1\}$,
 which means that $S_o(p)$ is not equal to $S_o(q)$.
\end{ex} 
In general, given a Gauss phrase $p$ with $k$th component $AxAy$, the
 result of a shift move on the $k$th component is a Gauss phrase $q$
 with $k$th component $xAyA$.
 Then $\gvector{l(A)}$ in $p$ is $\gvector{l(x)}$, the linking vector of
 the word $x$, and $\gvector{l(A)}$ in $q$ is $\gvector{l(y)}$, the
 linking vector of the word $y$.
Let $\cvector{l}$ be the linking vector of the $k$th component of
 $p$.
As the linking vector of a component is invariant under the shift move,
 $\cvector{l}$ is also the linking vector of the $k$th component of
 $q$.
Now, observe that
\begin{equation*}
 \cvector{l} = \gvector{l(x)} + \gvector{l(y)}
\end{equation*}
and therefore
\begin{equation}\label{eqn:linking-complement}
  \gvector{l(y)} = \cvector{l} + \gvector{l(x)}.
\end{equation}
\par
For any vector $\cvector{u}$ in $K_n$ we can define a map
$c_{\cvector{u}}$ from $K_n$ to itself by
\begin{equation*}
c_{\cvector{u}}(\cvector{v}) = \cvector{u} + \cvector{v}
\end{equation*}
for all vectors $\cvector{v}$ in $K_n$.
We define $K(\cvector{u})$ to be the set of orbits of $K_n$
under $c_{\cvector{u}}$.
Then Equation~\ref{eqn:linking-complement} implies that
$\gvector{l(x)}$ and $\gvector{l(y)}$ are in the same orbit of
$K_n(\cvector{l})$.
\par
Note that for the zero vector $\cvector{0}$ in $K_n$, $c_{\cvector{0}}$
is the identity map and each orbit in $K_n(\cvector{0})$ has one
element.
For any other vector $\cvector{u}$ in $K_n$, the map $c_{\cvector{u}}$
is an involution. 
That is, $c_{\cvector{u}} \circ c_{\cvector{u}}$ is
the identity map.
In this case, it is easy to check that each orbit in $K(\cvector{u})$
contains two elements.
We write $\vorbit{v}$ for the orbit of $\cvector{v}$ and
$\cvector{u}\in \vorbit{v}$ means that $\cvector{u}$ is in the orbit
of $\cvector{v}$.
\par
Using this idea of orbits, we adapt the definition of $S_o$ to define a
homotopy invariant of Gauss phrases.
\par
Let $p$ be a Gauss phrase and let $\svector{l}{k}$ be the linking vector
of the $k$th component of $p$.
Write $\mathcal{A}_k$ for the set of single-component letters appearing
in the $k$th component of $p$.
For an orbit $\vorbit{v}$ in $K(\svector{l}{k})$ we define a map $d_k$
from $K(\svector{l}{k})$ to $\Z$ by 
\begin{equation*}
 d_k(\vorbit{v}) = \sharp \lbrace X\in \mathcal{A}_k | \gvector{l(X)}
  \in \vorbit{v} \rbrace
\end{equation*}
where $\sharp$ means the number of elements in the set.
Let $O_k(p)$ be the subset of $K(\svector{l}{k}) - \{\vorbit{0}\}$
given by 
\begin{equation*}
 O_k(p) = \lbrace \vorbit{v} \in K(\svector{l}{k}) - \{\vorbit{0}\} |
  d_k(\vorbit{v}) \text{ is odd}\rbrace. 
\end{equation*}
\par
Note that we can calculate $O_k(p)$ from $B_k(p)$ as follows.
If $\svector{l}{k}$ is the zero vector, each orbit of
$K(\svector{l}{k})$ contains a single vector and $O_k(p)$ is the set of
orbits of vectors in $B_k(p)$.
If $\svector{l}{k}$ is not the zero vector, each orbit of
$K(\svector{l}{k})$ contains exactly two vectors.
In this case, if $\vorbit{u}$ is an orbit in
$K(\svector{l}{k}) - \{\vorbit{0}\}$ 
consisting of the vectors $\cvector{u}$ and $\cvector{v}$, $\vorbit{u}$
is in $O_k(p)$ if and only if exactly one of $\cvector{u}$ or
$\cvector{v}$ is in $B_k(p)$. 
\par
We define $S(p)$ to be the $n$-tuple where for each $k$ the $k$th element
is the pair $(\svector{l}{k}, O_k(p))$. We write $S_k(p)$ for
the $k$th element in $S(p)$.
\begin{thm}\label{thm:s_invariant}
Let $p$ be a Gauss phrase. Then $S(p)$ is a homotopy invariant of $p$.
\end{thm}
\begin{proof}
It is enough to prove that $S_k(p)$ is a homotopy invariant of $p$.
As we already know that the linking matrix is a homotopy invariant of a
 Gauss phrase, we know that the linking vector $\svector{l}{k}$ is
 invariant. So we just need to check what happens to the set $O_k(p)$
 under each move.
However, as we can derive $O_k(p)$ from $B_k(p)$, invariance of $O_k(p)$
 under the moves H1, H2 and H3 and isomorphism follows from
 Theorem~\ref{thm:so_invariant}.
Thus we just need to check what happens under a shift move applied to
 the $k$th component.
\par
For a shift move we must consider two cases.
First we consider what happens to a linking vector associated to a
 letter $A$ when a letter $B$ is moved by the shift move.
Given a Gauss phrase $p$ with $k$th
 component $BxAyAz$, the result of the shift move is a Gauss phrase $q$
 with $k$th component $xAyAzB$. However, in both cases the linking
 vector for $A$ is the linking vector for $y$. Thus the linking vector
 for $A$ is unchanged in this case and there is no effect on $O_k(p)$.
\par
Secondly we consider the case where $A$ itself is moved by the shift
 move. 
We have already observed above (Equation~\ref{eqn:linking-complement})
 that the linking vectors of $A$ before and after the shift move are in
 the same orbit of $K(\svector{l}{k})$.
Thus $O_k(p)$ is not affected by this change.
\par
Since $O_k(p)$ is unchanged in both cases, we can conclude that $O_k(p)$ is
 invariant under the shift move.
\end{proof}
We note that the linking matrix of a Gauss phrase $p$ can be
reconstructed from $S(p)$ by combining the linking vectors given as the
first elements of the pairs $S_k(p)$.
The following two examples shows that the invariant $S$ is stronger than
the linking matrix.
\begin{ex}\label{ex:calc_s_simple}
Let $p$ be the Gauss phrase $AB|AC|BC|\trivial$.
Then $\svector{l}{1}$ is $(0,1,1,0)$, $\svector{l}{2}$ is $(1,0,1,0)$,
$\svector{l}{3}$ is $(1,1,0,0)$ and $\svector{l}{4}$ is $(0,0,0,0)$. 
\par
As there are no single-component letters appearing in the Gauss phrase,
the sets $O_1(p)$, $O_2(p)$, $O_3(p)$ and $O_4(p)$ are all empty.
\par
The $S$ invariant of the Gauss phrase is given by 
$( (\svector{l}{1},\emptyset), (\svector{l}{2},\emptyset), (\svector{l}{3},\emptyset), (\svector{l}{4},\emptyset))$.
\end{ex}
\begin{ex}\label{ex:calc_s}
Let $q$ be the Gauss phrase
\begin{equation}\label{eqn:example_phrase}
ADBAEBCFCG|JLDHIHJK|EI|FGKL.
\end{equation}
Then $\svector{l}{1}$ is $(0,1,1,0)$, $\svector{l}{2}$ is $(1,0,1,0)$,
$\svector{l}{3}$ is $(1,1,0,0)$ and $\svector{l}{4}$ is $(0,0,0,0)$. 
These vectors are the same as those calculated in
 Example~\ref{ex:calc_s_simple}.
Thus the linking matrix of $q$ is equal to that of $p$ in
 Example~\ref{ex:calc_s_simple} and so the linking matrix invariant can
 not distinguish them.
\par
In order to calculate $O_1(q)$, we consider the single-component letters
 in the first component. 
The single-component letters are $A$, $B$ and $C$.
We see that $\gvector{l(A)}$ is $(1,1,0,0)$,
 $\gvector{l(B)}$ is $(1,0,1,0)$ and $\gvector{l(C)}$ is
 $(0,0,0,1)$. Since $\svector{l}{1}$ is non-zero, for each of the linking
 vectors of the letters, we need to determine the corresponding
 orbit in $K(\svector{l}{1})$. 
We find that $\gvector{l(A)}$ and $\gvector{l(B)}$ are in the same orbit
 and that $\gvector{l(C)}$ is in the same orbit as vector $(0,1,1,1)$. 
The only orbit occuring an odd number of times is $[(0,0,0,1)]$ and
 this orbit does not contain the zero vector. 
Therefore $O_1(q)$ is $\{[(0,0,0,1)]\}$.
\par
The second component only has two single-component letters $H$ and
 $J$. 
The linking vector for $H$, $\gvector{l(H)}$ is $(0,0,1,0)$. In
 $K(\svector{l}{2})$, this vector belongs to the orbit which
 also contains the vector $(1,0,0,0)$. 
The linking vector for $J$, $\gvector{l(J)}$ is $(1,0,1,1)$ which
 belongs to the orbit in $K(\svector{l}{2})$ which also
 contains $(0,0,0,1)$. Thus $\gvector{l(H)}$ and $\gvector{l(J)}$ belong
 to different orbits and neither belong to the orbit containing the zero
 vector. 
Therefore $O_2(q)$ is $\{[(0,0,1,0)],[(1,0,1,1)]\}$.
\par
The third and fourth components do not contain any single-component
 letters and so both $O_3(q)$ and $O_4(q)$ are empty.
\par
Since the sets $O_1(q)$ and $O_2(q)$ are not empty, we can see that they
 are not equivalent to the sets $O_1(p)$ and $O_2(p)$ 
 calculated in Example~\ref{ex:calc_s_simple}. 
This means that the
 invariant $S$ can distinguish $q$ from the Gauss phrase given in
 Example~\ref{ex:calc_s_simple}.
\end{ex}
As we did for the $S_o$ invariant, we provide a canonical way of writing
each $S_k(p)$ as a matrix.
We can then write $S(p)$ as a tuple of matrices. 
For two Gauss phrases $p$ and $q$, the equivalence of $S(p)$ and $S(q)$
is given by equality of the corresponding tuples of matrices.
\par
Suppose $p$ is an $n$-component Gauss phrase and $O_k(p)$ has $r$
elements.
We define a matrix $M_k$, with $n$ columns and $r+1$ rows, for $S_k(p)$
as follows.
We define the first row of the matrix to be $\svector{l}{k}$.
Then for each orbit $\vorbit{v}$ in $O_k(p)$, we define the
representative vector of $\vorbit{v}$ to be the smallest vector in
$\vorbit{v}$ (according to the order we defined on $K_n$).
The remainder of the rows in the matrix are then given by the
representative vectors of the members of $O_k(p)$ written out in
ascending order.
\par
Given a matrix constructed in this way we can easily reconstruct
$S_k(p)$. Indeed, the first row of the matrix gives $\svector{l}{k}$ and
the remaining rows give $O_k(p)$.
On the other hand, there is no room for choice in the construction of
the matrix and so, under this construction method, there is only one
matrix which represents $S_k(p)$.
\begin{ex}
Consider the Gauss phrase given in \eqref{eqn:example_phrase}.
We label it $q$. In Example~\ref{ex:calc_s} we
 calculated $S(q)$. Using the information we calculated there and the
 method explained above it is easy to represent $S(q)$ as a $4$-tuple
 of matrices: 
\begin{equation*}
\begin{pmatrix}
\begin{pmatrix}
0 & 1 & 1 & 0 \\
0 & 0 & 0 & 1 \\
\end{pmatrix},
\begin{pmatrix}
1 & 0 & 1 & 0 \\
0 & 0 & 0 & 1 \\
0 & 0 & 1 & 0 \\
\end{pmatrix},
\begin{pmatrix}
1 & 1 & 0 & 0 \\
\end{pmatrix},
\begin{pmatrix}
0 & 0 & 0 & 0 \\
\end{pmatrix}
\end{pmatrix}. 
\end{equation*}
\end{ex} 
\begin{rem}
As $S$ is invariant under isomorphism and the moves H1, H2 and H3 we can
 also consider it as an open homotopy invariant.
However, with reference to Example~\ref{ex:so_strong},
it is clear that combining the $S_o$ invariant with the
 linking matrix invariant gives a stronger open homotopy invariant than
 $S$.
\end{rem} 
%%%%%%%%%%%%%%%%%%%%%%%%%%%%%%%%%%%%%%%%%%%%%%%%
\section{Properties and realizability of $S$}\label{sec:s_conditions}
Suppose we are given an $n$-tuple of pairs $(\svector{l}{k},O_k)$, such
that $\svector{l}{k}$ is a vector in $K_n$ and $O_k$ is a subset of
$K(\svector{l}{k}) - \{[\cvector{0}]\}$.
In this section we give necessary and sufficient conditions for the
$n$-tuple $(\svector{l}{k},O_k)$ to be $S(p)$ for some Gauss phrase
$p$.
\par
We have already noted that we can reconstruct the linking matrix of a
Gauss phrase $p$ from $S(p)$.
Thus an obvious necessary condition is that the vectors $\svector{l}{k}$
should form a valid linking matrix.
\par
Suppose that $\cvector{u}$ is a vector in $K_n$ such that the $k$th
element of $\cvector{u}$ is $0$.
Then for any vector $\cvector{v}$ in $K_n$, the $k$th element of
$\cvector{v}$ and $c_{\cvector{u}}(\cvector{v})$ are the same. 
Thus we can say that an orbit in $K(\cvector{u})$ is $k$-odd if any
vector in the orbit is $k$-odd and this concept is well-defined. 
We define the term $k$-even for orbits in $K(\cvector{u})$ in an
analogous way.
\par
Let $p$ be a Gauss phrase and $\svector{l}{k}$ be the linking vector of
the $k$th component of $p$.
Note that because the entries on the main diagonal of the linking matrix
are all zero, the $k$th element of $\svector{l}{k}$ is $0$.
As $O_k(p)$ is a subset of $K(\svector{l}{k}) - \{[\cvector{0}]\}$ the
concept of $k$-odd orbit is well defined for members of $O_k(p)$.
The following proposition gives a necessary condition on $O_k$ to be
$O_k(p)$ for some Gauss phrase $p$. 
\begin{prop}\label{prop:s_necessary}
Let $p$ be a Gauss phrase.
Then $O_k(p)$ contains an even number of $k$-odd orbits.
\end{prop}
\begin{proof}
This follows from Proposition~\ref{prop:so_necessary} and the fact that
 $O_k(p)$ is derivable from $B_k(p)$.
\end{proof} 
The following proposition shows that the necessary conditions we have
given are also sufficient.
\begin{prop}\label{prop:s_sufficient}
Let $n$ be a positive integer and let $L$ be an $n \times n$ symmetric
 matrix with elements in $\cyclic{2}$ such that the entries on the main
 diagonal are all zero. 
Write $\svector{l}{k}$ for the $k$th row of $L$.
For each $k$ from $1$ to $n$ inclusive, let $O_k$ be a subset of
 $K(\svector{l}{k}) - \{[\cvector{0}]\}$ containing an even number of
 $k$-odd orbits. 
Then there exists a Gauss phrase $p$ such that $S_k(p)$ is equal to
 $(\svector{l}{k},O_k)$ for all $k$. 
\end{prop} 
\begin{proof}
For each orbit in $O_k$ we pick one of the vectors in the orbit as a
 representative vector.
How we chose the representative vector does not matter but we could do
 this by picking the smallest vector as we do when we write $S_k(p)$ as
 a matrix.
Then define $B_k$ to be the set of representative vectors of the orbits
 in $O_k$.
\par
By Proposition~\ref{prop:so_linking} we can construct a Gauss phrase $p$
 for which $B_k(p)$ is equal to $B_k$ for each $k$ and for which the
 linking matrix is equal to $L$.
It is then easy to see that $O_k(p)$ is equal to $O_k$ for all $k$.
Thus $p$ is the required Gauss phrase.
\end{proof}
\begin{rem}
As for the invariant $S_o$, we note that the number of values that $S$
 takes over the set of $n$-component Gauss phrases is finite.
\end{rem}
%%%%%%%%%%%%%%%%%%%%%%%%%%%%%%%%%%%%%%%%%%%%%%%%
\section{Other homotopies}\label{sec:other_homotopy}
We finish this paper with some remarks about some other kinds of
homotopy.
\par
Recall that when we constructed a Gauss phrase from a virtual link
diagram we assumed that an order had been assigned to the components in
the link. 
Permuting the order of the components of the link may produce a
different Gauss phrase.
With this in mind, we define a new move on Gauss phrases called
\emph{component permutation} which allows us to permute the components
of a Gauss phrase. 
By allowing this kind of move we can introduce a new kind of homotopy
which we call \emph{unordered homotopy}.  
Under this kind of homotopy we do not attach any significance to the
order of the components in a Gauss phrase. 
Invariants introduced in this paper become invariants
of unordered homotopy by considering equivalence of invariants modulo
some kind of operation involving permutation.
\par
We now consider the general case of nanophrases.
Recall that a nanophrase is a pair consisting of a Gauss phrase and a
map.
The map, called a projection, maps the letters appearing in the Gauss
phrase to some set $\alpha$ which is fixed.
We can define a map from the set of nanophrases to the set of Gauss
phrases by forgetting the projection.
It is easy to check that this induces a well-defined map from
equivalence classes of nanophrases under a particular homotopy, to
equivalence classes of Gauss phrases under homotopy.
Thus homotopy invariants of Gauss phrases are invariants for any
homotopy of nanophrases.
\par
In particular, the $S$ invariant is invariant under any nanophrase
homotopy.
However, using extra information contained in the projection, it should
be possible to make stronger nanophrase homotopy invariants based on
the construction of $S$.
Equally, it should be possible to construct stronger open homotopy
invariants of nanophrases based on $S_o$.
We leave consideration of such possibilities to another paper.
%%%%%%%%%%%%%%%%%%%%%%%%%%%%%%%%%%%%%%%%%%%%%%%%
\bibliography{mrabbrev,gaussphrase}
\bibliographystyle{hamsplain}
\end{document}